\documentclass{IEEEtran}
\ifCLASSINFOpdf
   \usepackage[pdftex]{graphicx}
\else
   \usepackage[dvips]{graphicx}
\fi
\usepackage{xcolor}
\usepackage[cmex10]{amsmath}
\usepackage{amsfonts, amsthm}
\usepackage{bm}
\usepackage{url}
\makeatletter
\let\old@ps@headings\ps@headings
\let\old@ps@IEEEtitlepagestyle\ps@IEEEtitlepagestyle
\def\psccfooter#1{%
    \def\ps@headings{%
        \old@ps@headings%
        \def\@oddfoot{\strut\hfill#1\hfill\strut}%
        \def\@evenfoot{\strut\hfill#1\hfill\strut}%
    }%
    \def\ps@IEEEtitlepagestyle{%
        \old@ps@IEEEtitlepagestyle%
        \def\@oddfoot{\strut\hfill#1\hfill\strut}%
        \def\@evenfoot{\strut\hfill#1\hfill\strut}%
    }%
    \ps@headings%
}
\makeatother

\newtheorem{proposition}{Proposition}

\newcommand{\grad}[2]{\nabla_{#2}{#1}}

\newcommand{\hhess}[2]{\nabla^2_{#2}{#1}}

\newcommand{\refsec}[1]{Section~\ref{#1}}

\newcommand{\diag}[1]{\lfloor \!\!\lfloor #1 \rfloor\!\!\rfloor}

\definecolor{mygreen}{RGB}{80,200,120}

\begin{document}
%
\title{A Feasible Reduced Space Method for Real-Time Optimal Power Flow}

\author{
\IEEEauthorblockN{
 François Pacaud, Daniel Adrian Maldonado, Sungho Shin, Michel Schanen, Mihai Anitescu
}
\IEEEauthorblockA{
 Mathematics and Computer Science Department \\
 Argonne National Laboratory\\  Lemont, U.S.A \\
\{fpacaud,maldonadod,sshin,mschanen,anitescu\}@anl.gov}
}

\maketitle

\begin{abstract}
  We propose a novel feasible-path algorithm to solve the optimal
  power flow (OPF) problem for real-time use cases. The method augments the seminal work
  of Dommel and Tinney with second-order derivatives to work directly
  in the reduced space induced by the power flow equations.
  In the reduced space, the optimization problem includes only
  inequality constraints corresponding to the operational constraints.
  While the reduced formulation directly enforces the physical constraints,
  the operational constraints are softly enforced through Augmented Lagrangian penalty
  terms. In contrast to interior-point algorithms (state-of-the art for
  solving OPF), our algorithm maintains feasibility at each iteration, which
  makes it suitable for real-time application.
  By exploiting
  accelerator hardware (Graphic Processing Units) to compute the reduced Hessian, we show that the second-order method is numerically tractable and is effective to solve both static and real-time OPF problems.
\end{abstract}


\section{Introduction}
With the increasing penetration of rapidly varying renewable generation resources and electrical vehicles, there is a growing need to compute the generation dispatch at much higher time frequency.
This requires adaptation of optimal power flow (OPF) algorithms for operation in a real-time setting.
Here, finding the optimum becomes secondary comparing to finding a feasible solution
within a tight time constraint~\cite{Capitanescu2016}.
In addition, state-of-the-art OPF algorithms are mature tools on a slow timescale,
but they are not adapted to operate in a real-time setting. Generally, the
intermediate iterates do not satisfy the power flow equations --- encoding
the physical constraints (PCs) of the network --- and the solution is not realizable
before the algorithm has converged.
Novel real-time OPF approaches try to remedy this issue and track
closely the network changes, at a faster timescale~\cite{tang2017real,gan2016online,hauswirth2017online,dall2016optimal}.

The reduced space algorithm of Dommel and Tinney~\cite{dommel_optimal_1968}
is a promising candidate for these real-time applications. By design, it works directly
in the manifold induced by the power flow equations, so all iterates inherently satisfy
the PCs.The operational constraints (OCs) are enforced with soft penalties,
commonly used in real-time optimization to avoid expensive active-set reordering operations ~\cite{tang2017real,hauswirth2017online,zavala2010real,chiang2017augmented}.
However, the algorithm has fallen out of favor in the 1980s,
with the advent of interior point algorithms (see~\cite{frank_optimal_2012}
for a survey). This is due to practical limitations of the algorithm:
(1) the iterates are updated in the reduced-space using
only first-order information, impairing the speed of convergence of the algorithm.
(2) the large number of OCs in the problem
limits the available degrees of freedom in the reduced space.
We refer to the recent report of Kardos et al.~\cite{kardos2020reduced} for further details
on the limitations of the algorithm.

This article revisits the reduced space algorithm~\cite{dommel_optimal_1968}
in a real-time optimization setting while addressing (1) and (2).
(1) We efficiently extract the reduced Hessian in the reduced space, by leveraging GPUs
with automatic differentiation (AD) and by parallelizing the code.
The GPUs enable fast computation of second-order information to compute the descent direction in the reduced space.
(2) We propose an augmented Lagrangian (AL) formulation to reformulate the OCs 
with smooth penalty terms (PCs being satisfied by design).
First, this allows to devise a practical AL algorithm
to solve the \emph{static OPF} in the reduced space, with a tractable running time.
Second, we can exploit the AL formulation to track
the solution found previously by the static OPF algorithm, in a real-time setting.
Our \emph{real-time OPF} algorithm explicitly exploits the reduced Hessian to update
the tracking control. Previous real-time algorithms~\cite{tang2017real,hauswirth2017online}
only rely on first-order information (with a projected gradient or a LBFGS algorithm).
Numerical results show that the method efficiently tracks a suboptimal solution on instances with up to 9241 buses.

In \refsec{sec:model}, we introduce
the OPF problem before presenting in \refsec{sec:projection}
the reduced space problem with the expressions
of the gradient and the Hessian in the reduced space.
In \refsec{sec:algorithm}, we develop two
algorithms based on an AL formulation to
solve static and real-time OPF problems.
Finally, we show in Section~\ref{sec:results} that the two algorithms are numerically tractable by solving large-scale OPF instances.

\subsection{Notation}

All vectors
are noted in bold $\bm{x} = (x_1, \cdots, x_n) \in \mathbb{R}^n$.
$\|\cdot\|$ and $\diag{\cdot}:\mathbb{R}^n \to \mathbb{R}^{n \times n}$ denote the Euclidean norm and the diagonal operator, respectively.
For any variable $x$, its lower bound is denoted by $x^\flat$, its upper
bound by $x^\sharp$.
Finally, $d$ denotes the differential operator, $d_x(\cdot)$ the total derivative $\frac{d(\cdot)}{dx}$,
and $\partial_x(\cdot)$ the partial derivative $\frac{\partial (\cdot)}{\partial x}$.
Consistent with the Jacobian's definition,
the gradient is defined as the \emph{row} vector $\nabla_x(\cdot) =
\big(\partial_{x_1} (\cdot), \cdots, \partial_{x_n}(\cdot)\big)\in \mathbb{R}^{1 \times n}$,
and is thus assimilated into the total derivative.

\section{Model and formulation}
\label{sec:model}
This section introduces an alternate current optimal power flow model (AC-OPF) associated
with a transmission grid.

\subsection{Optimal power flow}
We use the most widely used polar formulation \cite{cain2012history} as our model.
Let a power grid with $n_b$ buses, $n_g$ generators and $n_\ell$
lines. We denote by $\bm{v}, \bm{\theta} \in \mathbb{R}^{n_b}$ the voltage magnitudes
and the voltage angles at each bus. The active and reactive power generations
are denoted respectively by $\bm{p}^g, \bm{q}^g \in \mathbb{R}^{n_g}$,
and the active and reactive loads by $\bm{p}^d, \bm{q}^d$.

At fixed loads $(\bm{p}^d, \bm{q}^d)$,
the optimal power flow aims at finding an operational point
$(\bm{v}, \bm{\theta}, \bm{p}^g, \bm{q}^g)$
that minimizes the active power generation cost while satisfying the
physical and the operational constraints of the network.
We get the ACOPF problem written in polar form:
\begin{subequations}
  \label{eq:acopf}
  \begin{align}
      \label{eq:acopf:obj}
      \min_{\bm{v}, \bm{\theta}, \bm{p}, \bm{g}} & \; \sum_{i=1}^{n_g} c^g_{i,1} (p_i)^2 + c_{i,2}^g (p_i) + c_{i, 3}^g \\
      \label{eq:acopf:cons}
      \text{subject to } ~ & G(\bm{v}, \bm{\theta}, \bm p, \bm q) = 0 \, \quad  H(\bm v, \bm \theta) \leq \bm{h}^\sharp \, , \\
      \label{eq:acopf:bounds1}
                           & \bm{v}^\flat \leq \bm{v} \leq \bm{v}^\sharp\, , \quad \theta_{ref} = 0 \; , \\
      \label{eq:acopf:bounds2}
                    & \bm{p}^\flat \leq \bm{p} \leq \bm{p}^\sharp \; , ~ \bm{q}^\flat \leq \bm{q} \leq \bm{q}^\sharp \; .
  \end{align}
\end{subequations}
The quadratic objective~\eqref{eq:acopf:obj} depends only on the active power generation,
with coefficients $c^g_1, c^g_2, c^g_3 \in \mathbb{R}^{n_g}$ specifying the cost of each generator.
In~\eqref{eq:acopf:cons}, the equality constraint $G$ encodes the $2\times n_b$ complete power balance equations
and the inequality constraint $H$ encodes the $2 \times n_\ell$ line flow limits.
The bounds~\eqref{eq:acopf:bounds1}-\eqref{eq:acopf:bounds2}
ensure that the voltage magnitudes and the power generations satisfy their operational limits.

\subsection{State and control variables}
In the power flow, the buses are classified into three categories: REF (or slack bus),
PV (or generator buses), and PQ (or load buses).
We define the \emph{control} variable $\bm{u}$ by gathering
the voltage magnitudes at the PV buses and at the slack, as well as the active power generation at PV buses.
Similarly, we define a \emph{state} variable $\bm x$ with the voltage magnitudes at PQ buses
and the voltage angles at PQ and PV buses. We get:
\begin{equation}
  \label{eq:statecontrol}
  \bm{u} = (\bm v_{ref}, \bm{v}_{pv}, \bm{p}_{pv}^g) \; , \quad
  \bm{x} = (\bm{\theta}_{pv}, \bm{\theta}_{pq}, \bm{v}_{pq}) \; .
\end{equation}
If the control $\bm u$ is fixed, the state $\bm x$
is entirely determined by a subset of the power flow equations $G$,
denoted here by the function $g:\mathbb{R}^{n_x} \times \mathbb{R}^{n_u} \to \mathbb{R}^{n_x}$~\cite{tinney1967power}.

Hence, we can derive a new OPF formulation depending only
on the state $\bm x$ and the control $\bm u$. We define the functions
$h(\bm x, \bm u) = H(\bm v, \bm \theta)$ for line constraints,
and $r(\bm x, \bm u) = \big(\bm{v}_{pq}, \bm{p}_{ref}^g, \bm{q}^{net}_{ref}, \bm{q}^{net}_{pv}\big)
\in \mathbb{R}^{n_{pq}  + 2 n_{ref}+ n_{pv}}$ to gather the remaining operational constraints
(voltage angles are unconstrained), yielding the problem
\begin{equation}
  \label{eq:acopf_controlstate}
  \begin{aligned}
    \min_{\bm x, \bm u} \quad & f(\bm x, \bm u)                                            \\
    \text{subject to} \quad   & \bm{u}^\flat \leq \bm{u} \leq \bm{u}^\sharp  \;, \quad g(\bm x, \bm u) = 0 \; , \\
                              &\bm{r}^\flat \leq r(\bm x, \bm u) \leq \bm{r}^\sharp \; , \quad
                              h(\bm x, \bm u) \leq \bm{h}^\sharp \; , \\
  \end{aligned}
\end{equation}
with $\bm{r}^\flat = (\bm{v}_{pq}^\flat, \bm{p}_{ref}^{g, \flat}, \bm{q}_{ref}^{g, \flat}, C_g^{pv} \bm{q}_{pv}^{g, \flat})$
and $\bm{r}^\sharp = (\bm{v}_{pq}^\sharp, \bm{p}_{ref}^{g, \sharp}, \bm{q}_{ref}^{g, \sharp}, C_g^{pv} \bm{q}_{pv}^{g, \sharp})$
(by convention the slack bus has only one generator). $C_g \in \mathbb{R}^{n_b \times n_g}$ is the bus-generator incidence matrix. In \eqref{eq:acopf_controlstate}, we
have a total of $m = 2n_\ell + n_{pq} + n_{pv} + 2$ nonlinear constraints.

Note that problem~\eqref{eq:acopf_controlstate} is not equivalent to
the original OPF~\eqref{eq:acopf}, as we are not controlling explicitly the reactive power generations
(we only bound the net power injections in the functional $r(\bm x, \bm u)$). That means that
if we have multiple generators associated with the same bus, we cannot recover
the individual reactive power generations for each generator from a couple $(\bm x, \bm u)$. However,
when compared to the structure of~\eqref{eq:acopf}, \eqref{eq:acopf_controlstate}
allows us to significantly reduce the
dimension of the problem. We will explain in the next section the basis of the reduced space
method.

\section{Reduced space problem}
\label{sec:projection}
Section~\ref{sec:model} introduced the OPF problem, parameterized by a state $\bm x$ and a control $\bm u$.
In \refsec{sec:projection:reduced_space}, we exploit the implicit relation between the control $\bm u$ and the state $\bm x$ to build a reduced space problem, depending
\emph{only} on the control $\bm{u}$. Then, we derive in \refsec{sec:projection:reduced_gradient} the reduced gradient
and the reduced Hessian using the adjoint and the adjoint-adjoint methods, respectively.

\subsection{Reduced space}
\label{sec:projection:reduced_space}
The functional $g$ encoding the $n_x$ power flow equations
is continuously differentiable.
By the Implicit Function
theorem, it follows that if the Jacobian $\nabla_{x} g(\cdot, \bm{u})$ is invertible
at a given control $\bm u$, then there exists
a local function $x: \mathbb{R}^{n_u} \to \mathbb{R}^{n_{x}}$ such that
$g(x(\bm{u}), \bm{u}) = 0$ locally.

\paragraph{Resolution of power flow equations}
At a fixed control~$\bm u$, the nonlinear equations
$g(\bm{x}, \bm{u}) = 0$ are solved using a Newton-Raphson
algorithm~\cite{tinney1967power}. Starting at an initial guess $\bm{x}^{(0)}$, the
algorithm computes the solution $x^*(\bm u)$ through
\begin{equation}
  \label{eq:projection:newtonraphson}
  \bm{x}_{k+1} = \bm{x}_{k} -
  (\nabla_{x} g_k)^{-1} \bm{g}_k
  \; .
\end{equation}

\paragraph{Reduced space OPF}
Applying the Implicit Function theorem, we eliminate all equality constraints in \eqref{eq:acopf_controlstate}
and we get the reduced space problem with remaining control variables $\bm{u}$:
\begin{equation}
  \label{eq:reduced_space_problem}
  \begin{aligned}
    \min_{\bm{u}^\flat \leq \bm u \leq \bm{u}^\sharp} \quad & f(x(\bm u), \bm u) \\
    \text{subject to} \quad  & h(x(\bm u), \bm u) \leq \bm{h}^\sharp \; , ~
    \bm{r}^\flat \leq r(x(\bm{u}), \bm{u}) \leq \bm{r}^\sharp \; .
  \end{aligned}
\end{equation}
Problem~\eqref{eq:reduced_space_problem} is equivalent to \eqref{eq:acopf_controlstate},
but has a smaller dimension ($n_u$ instead of $n_u + n_x$).
This implies solving the power flow equations
for all trial points $\bm u$, and requires expensive (but highly parallelizable)
computations for the reduced gradient and the reduced Hessian (see \refsec{sec:projection:reduced_gradient}).

\subsection{Reduced sensitivities}
\label{sec:projection:reduced_gradient}

\subsubsection{Reduced gradient}
If we use the chain rule to differentiate the objective in \eqref{eq:reduced_space_problem},
we get $\nabla_{\bm u} f = \partial_{\bm u} f + \partial_{\bm x} f \cdot \nabla_{\bm u} x$.
However, this expression is expensive to evaluate,
as the Jacobian $\nabla_{\bm u} x$ (also called \emph{sensitivity matrix})
is a dense matrix with dimension $n_x \times n_u$.
Instead, we evaluate the reduced gradient with the \emph{adjoint method}.
\begin{proposition}[Adjoint method]
  \label{prop:adjoint}
  Let $(\bm x, \bm{u}) \in \mathbb{R}^{n_x} \times \mathbb{R}^{n_x}$ such that $g(\bm{x}, \bm{u}) = 0$,
  and a first-order adjoint $\bm{\lambda} \in \mathbb{R}^{n_x}$ solution of
  the linear system $(\nabla_{\bm x} g)^\top \bm{\lambda} = - (\partial_{\bm x} f)^\top$.
  Then, the reduced gradient satisfies
  \begin{equation}
    \label{eq:reduced_gradient}
    \nabla_{\bm{u}} f =
    \partial_{\bm{u}} f +
    \bm{\lambda}^\top \cdot \nabla_{\bm{u}} g
    \; .
  \end{equation}
\end{proposition}
\begin{proof}
  Let $(\bm x, \bm u) \in \mathbb{R}^{n_x} \times \mathbb{R}^{n_u}$
  such that $g(\bm x, \bm u) = 0$.
  For all $\bm{\lambda} \in \mathbb{R}^{n_x}$, we define the Lagrangian
  \begin{equation}
    \label{eq:lagrangian}
  \ell(\bm{x}, \bm{u}, \bm{\lambda}) := f(\bm{x}, \bm{u}) + \bm{\lambda}^\top g(\bm{x}, \bm{u})
  \; .
  \end{equation}
  As $g(\bm{x}, \bm{u}) = 0$, the Lagrangian $\ell(\bm{x}, \bm{u}, \bm{\lambda})$
  does not depend on $\bm{\lambda}$ and we have $\ell(\bm{x}, \bm{u}, \bm{\lambda}) = f(\bm x, \bm u)$.
  Applying the chain-rule, leads to the total derivative of $\ell$
  w.r.t. $\bm{u}$
  \begin{equation*}
    \begin{aligned}
      d_{\bm u} \ell &= \big(\partial_{\bm{x}} f \cdot \nabla_{\bm{u}} x +
       \partial_{\bm{u}} f \big) +
       \bm{\lambda}^\top \big(\nabla_{\bm x} g \cdot \nabla_{\bm u} x +
         \nabla_{\bm u} g
       \big) \\
              &= \big(\partial_{\bm{u}} f + \bm{\lambda}^\top \nabla_{\bm{u}} g \big) +
              \big(\partial_{\bm{x}} f + \bm{\lambda}^\top\nabla_{\bm{x}} g \big) \nabla_{\bm u} x \; .
    \end{aligned}
  \end{equation*}
  Fixing $\bm{\lambda} = - (\nabla_{\bm{x}}g)^{-\top}(\partial_{\bm{x}} f)^\top$,
  eliminates the dependency w.r.t. $\nabla_{\bm u} x$
  and we get the expression in \eqref{eq:reduced_gradient}.
\end{proof}

\subsubsection{Reduced Hessian}
\label{sec:projection:reduced_hessian}

By analogy to Proposition~\ref{prop:adjoint}, we derive the adjoint-adjoint method~\cite{wang1992second}, which instead of using the nonlinear equations $g(\bm{x}, \bm{u}) = 0$, considers the extended nonlinear system
\begin{equation}
  \label{eq:first_order}
  \left\{
  \begin{aligned}
    & g(\bm x, \bm u) = 0 \\
    & \partial_{\bm{x}}{f}(\bm{x}, \bm{u}) + \bm{\lambda}^\top \grad{g}{\bm{x}}(\bm{x}, \bm{u}) = 0 \; ,
  \end{aligned}
  \right.
\end{equation}
(the second line is null by definition of the first-order adjoint $\bm{\lambda}$ in Proposition~\ref{prop:adjoint}).
By associating two adjoints $\bm{z}, \bm{\psi}$ to the two equations~\eqref{eq:first_order},
the adjoint-adjoint method amounts to the Hessian-vector product $(\nabla^2 f) \bm{w}$, $\bm{w} \in \mathbb{R}^{n_u}$.
\begin{proposition}[Adjoint-adjoint method~\cite{wang1992second}]
  \label{prop:secondorderadjoint}
  Let $(\bm x, \bm u, \bm \lambda) \in \mathbb{R}^{n_x} \times \mathbb{R}^{n_u} \times \mathbb{R}^{nx}$
  such that \eqref{eq:first_order} is satisfied, and $\bm w \in \mathbb{R}^{n_u}$
  any real vector.
  Then, the reduced Hessian-vector product is equal to
  \begin{equation}
    \label{eq:hessvecprod}
    (\nabla^2 f) \bm{w} = (\hhess{\ell}{\bm{u}\bm{u}}) \, \bm{w}
    +  (\hhess{\ell}{\bm{u}\bm{x}})^\top \bm{z}
    +  (\grad{g}{\bm{u}})^\top\bm{\psi}
    \;.
  \end{equation}
  where the two second-order adjoints $(\bm z, \bm \psi) \in \mathbb{R}^{n_x} \times \mathbb{R}^{n_x}$
  are defined as solutions of the two linear systems:
  \begin{equation}
    \label{eq:socadjoint}
    \left\{
    \begin{aligned}
      & (\grad{g}{\bm{x}})\phantom{^\top} \bm{z} = - \nabla_{\bm{u}} g \cdot \bm{w} \\
      & (\grad{g}{\bm{x}})^\top \bm{\psi} =
      - (\hhess{\ell}{\bm{x}\bm{u}}) \bm{w} \;
      - (\hhess{\ell}{\bm{x}\bm{x}}) \bm{z} \; .
    \end{aligned}
    \right.
  \end{equation}
\end{proposition}
\begin{proof}
  Let $\hat{g}(\bm{x}, \bm{u}, \bm{\lambda}) :=  \partial_{\bm{x}} f(\bm{x}, \bm{u}) + \bm{\lambda}^\top \grad{g}{\bm{x}}(\bm{x}, \bm{u})$.
  We define a new Lagrangian $\hat{\ell}$ associated with
by introducing two second-order adjoints $\bm{z}, \bm{\psi} \in \mathbb{R}^{n_x}$:
\begin{equation*}
  \hat{\ell}(\bm{x}, \bm{u}, \bm{w}, \bm{\lambda}; \bm{z}, \bm{\psi}) :=
  (\nabla_{\bm{u}} \ell)^\top \bm{w} + \bm{z}^\top g(\bm{x}, \bm{u})
  + \bm{\psi}^\top \hat{g}(\bm{x}, \bm{u}, \bm{\lambda})
  \; .
\end{equation*}
When~\eqref{eq:first_order} is satisfied, $\hat{l}$ does not depend
on $\bm z$ and $\bm \psi$.
By deriving $\hat{\ell}$ and choosing $(\bm z, \bm \psi)$ solutions
of the two linear systems~\eqref{eq:socadjoint}, we get
the expression~\eqref{eq:hessvecprod}.
\end{proof}

As $\nabla^2 \ell = \nabla^2 f + \bm{\lambda}^\top \nabla^2 g$,
we observe that both Equations~\eqref{eq:hessvecprod}
and \eqref{eq:socadjoint} involve the third-order tensors
$\nabla^2_{\bm{x}\bm{x}} g,$ $\nabla^2_{\bm{x}\bm{u}} g
,$ and $\nabla^2_{\bm{u}\bm{u}} g$. We will see in the next subsection
that in practice we do not need to evaluate the three tensors explicitly.

\subsection{Reduced callbacks}
\label{sec:projection:discussion}
We have introduced in Sections~\ref{sec:projection:reduced_space} and
\ref{sec:projection:reduced_gradient} all the elements to compute the
callbacks for the reduced space problem~\eqref{eq:reduced_space_problem}.

\paragraph{Objective and constraints}
Evaluating the objective $f(x(\bm{u}), \bm{u})$ and the constraints
$h(x(\bm{u}), \bm{u})$ requires the evaluation of $x(\bm{u})$ with
the Newton-Raphson algorithm~\eqref{eq:projection:newtonraphson}.
Typically, the algorithm converges in a few iterations.
At iteration $k$, the algorithm amounts to (i) evaluate the
sparse Jacobian $\nabla_{\bm{x}} g_k = \nabla_{\bm{x}} g(\bm{x}_k, \bm{u})$  (ii) solve
the linear system $(\nabla_{\bm{x}} g_k) \bm{d}_k = - \bm{g}_k$
to find the descent direction $\bm{d}_k$. In step (i), we evaluate
the sparse Jacobian $\nabla_{\bm{x}} g_k$ using forward-mode automatic
differentiation. In step (ii), we can use any sparse linear solver.

\paragraph{Reduced gradient and Jacobian}
The reduced gradient $\nabla_{\bm{u}} f$ requires the evaluation
of the gradients $\partial_{\bm{x}} f, \partial_{\bm{u}} f$ and the Jacobian $\nabla_{\bm{u}} g$.
The former are evaluated with manual adjoint differentiation,
the latter with forward-mode automatic differentiation. The linear system
$(\nabla_{x} g)^\top \bm{\lambda} = - (\partial_{\bm{x}} f)^\top$ is solved
with the sparse linear solver used in the power flow algorithm (same matrix, but transposed).
Similarly, we evaluate the reduced Jacobians $(\nabla_{\bm{u}} h, \nabla_{\bm{u}} r)$, except
that it requires solving $m$ linear systems to evaluate the intermediate
adjoints $\bm{\mu}_i$, with, for all $i=1, \cdots, m$: $(\nabla_{\bm{x}} g)^\top \bm{\mu}_i =
- (\partial_{\bm{x}} h_i)^\top$.

\paragraph{Reduced Hessian}
The reduced Hessian of the objective $\nabla^2 f$ is evaluated with the adjoint-adjoint method,
using $n_u$ Hessian-vector products $\nabla^2 f \cdot \bm{e}_i$ (with $\bm{e}_i$ the $i$-th
Cartesian basis vector).
In \eqref{eq:hessvecprod}, we avoid evaluating explicitly the tensors
$\nabla^2_{\bm{x}\bm{x}} g,$ $\nabla^2_{\bm{x}\bm{u}} g
,$ and $\nabla^2_{\bm{u}\bm{u}} g$ by computing inplace the forward-over-reverse
projection $\sum_{i=1}^{n_x} \lambda_i (\nabla^2 g_i) \bm{w}$.
In total, the dense reduced Hessian $\nabla^2 f$ requires solving $2n_u$ sparse linear systems , while factorizing both sparse matrices in \eqref{eq:socadjoint} only once.

However, evaluating the Hessian of the constraints $\nabla^2_{\bm{u}\bm{u}} h$
is generally not tractable. The adjoint-adjoint procedure (see Proposition~\ref{prop:secondorderadjoint})
would be repeated
for each first-order adjoint $\mu_i$, with $i=1,\cdots,m$ and would involved solving $2 n_u \times m$ linear systems.
In the next section, we will reformulate all operational constraints as
soft AL penalties. By doing so, we will only need
to evaluate the reduced Hessian of the AL functional.

\section{Resolution algorithms}
\label{sec:algorithm}

In the previous section, we have devised a method to compute the reduced gradient,
the reduced Jacobian and the reduced Hessian of the reduced space problem
\eqref{eq:reduced_space_problem}. In~\ref{sec:algo:auglag}, we reformulate the
reduced space problem with an AL formulation.
We exploit the AL formulation to design two algorithms,
both working in the reduced space. The first algorithm, introduced in \refsec{sec:algo:static_opf},
solves the static OPF with an AL algorithm.
The second algorithm, in \refsec{sec:algo:rto_opf}, is able to track the
solution of a real-time OPF problem.

\subsection{Augmented Lagrangian formulation}
\label{sec:algo:auglag}
In what follows, we gather all inequalities
in a single functional $c: \mathbb{R}^{n_u} \to \mathbb{R}^{m}$: $\bm{s}^\flat \leq c(\bm{u}) \leq \bm{s}^\sharp$,
with $c(\bm u) = \big(h(\bm u), r(\bm u)\big)$,
$\bm{s}^\flat = \big(\bm{0}_{2n_l}, \bm{r}^\flat\big)$ and
$\bm{s}^\sharp = \big(\bm{h}^{\sharp}, \bm{r}^\sharp\big)$.
We rewrite Problem~\eqref{eq:reduced_space_problem} in standard form by
introducing a slack vector $\bm{s}^\flat \leq \bm{s} \leq \bm{s}^\sharp$
satisfying $c(\bm u) - \bm s = 0$:
\begin{equation}
  \label{eq:problem_algo}
  \min_{
    \substack{
  \bm{u}^\flat \leq \bm u \leq \bm{u}^\sharp, \\
  \bm{s}^\flat \leq \bm s \leq \bm{s}^\sharp
    }
  } \; f(\bm u)
  \quad \text{subject to} \quad c(\bm u) - \bm s =  0  \; .
\end{equation}
Problem~\eqref{eq:problem_algo} can be solved with an interior-point method (IPM). However,
the evaluation of the Hessian of the Lagrangian requires
the resolution of $(m +1) \times (2 n_u + 1)$ linear systems (see \refsec{sec:projection:discussion}),
which is quickly prohibitive. To alleviate this, both Quasi-Newton
and constraint aggregation schemes can be used, so far with mixed results.
We refer to \cite{kardos2020reduced} for a detailed discussion about
the resolution of the reduced problem~\eqref{eq:problem_algo} with IPM.

By contrast, by moving all inequality constraints into the objective,
an AL formulation requires the resolution of only $(2 n_u +1)$
linear systems to evaluate the full reduced Hessian.
For a given penalty $\rho > 0$ and multiplier vector $\bm y \in \mathbb{R}^m$,  the 
AL subproblem associated with Problem~\eqref{eq:problem_algo} states
\begin{equation}
  \label{eq:auglag_subproblem}
  \min_{
    \substack{
  \bm{u}^\flat \leq \bm u \leq \bm{u}^\sharp, \\
  \bm{s}^\flat \leq \bm s \leq \bm{s}^\sharp
    }
  } \; f(\bm u) + \bm{y}^\top \big( c(\bm u) - \bm s\big) +
  \frac{\rho}{2} \| c(\bm u) - \bm s \|^2 \; .
\end{equation}

Note that introducing the slack variable $\bm{s}$ increases the dimension of the problem
from $n_u$ to $n_u + m$, with $m$ potentially a large number.
However, it is well known that the slack variable depends implicitly on the control $\bm{u}$
\cite{rockafellar1976augmented,bertsekas1976multiplier}. Unfortunately,
removing the slack $\bm{s}$ in \eqref{eq:auglag_subproblem} leads to an optimization problem
with discontinuous second-order derivatives, impairing the solution algorithm.
To avoid this, we exploit instead the structure of the KKT system and show that
the slack descent direction $\bm{d}_s$ depends linearly on the control descent direction $\bm{d}_u$.

\subsubsection{Callbacks}
we define the AL functional as:
$L_\rho(\bm u, \bm s; \bm y) = f(\bm u) + \bm{y}^\top \big( c(\bm u) - \bm s\big) +
  \frac{\rho}{2} \| c(\bm u) - \bm s \|^2$ yielding
\begin{equation}
  \label{eq:gradient_auglag}
  \left\{
    \begin{aligned}
      & \nabla_{\bm u} L_\rho(\cdot) = \nabla_{\bm{u}} f(\bm{u}) + \big(\bm{y} + \rho(c(\bm{u}) - \bm{s})\big)^\top \nabla_{\bm{u}} c(\bm{u}) \\
      & \nabla_{\bm s} L_\rho(\cdot) = - \big(\bm{y} + \rho (c(\bm{u}) - \bm{s})\big)^\top
      \; .
    \end{aligned}
  \right.
\end{equation}
We note that evaluating the gradient of the AL problem
involves the reduced gradient $\nabla_{\bm u} f$ and a Jacobian-transpose
vector product $\bm{v}^\top \nabla_{\bm{u}} c$ (both efficiently
computed with our adjoint implementation).

By differentiating again~\eqref{eq:gradient_auglag}, we get the
Hessian of the AL functional:
\begin{equation}
  \label{eq:hessian_auglag}
  \nabla^2 L_\rho =
  \begin{bmatrix}
    H_{\bm{u}\bm{u}} + \rho (\nabla_{\bm{u}} c)^\top \nabla_{\bm{u}} c  & - \rho (\nabla_{\bm{u}} c)^\top  \\
    - \rho \nabla_{\bm{u}} c  & \rho I
  \end{bmatrix}
  \; ,
\end{equation}
where $H_{\bm{u}\bm{u}} = \nabla^2 f(\bm{u}) + \sum_{i=1}^m \big(y_i + \rho (c_i(\bm{u}) - s_i)\big) \nabla^2 c_i(\bm{u})$.

\subsubsection{Scaling}
\label{sec:algo:scaling}
In our implementation, we have to scale the problem to ensure that the order
of magnitude of the objective matches those of the different constraints (and thus
avoid any degeneracy). The scaling of the objective $\sigma_f \in \mathbb{R}$
and the scaling of the constraints $\bm{\sigma}_c \in \mathbb{R}^m$ can be heuristically estimated, or by scaling the constraints with the absolute norm of their gradient
\cite{wachter2006implementation,birgin2014practical}.
By noting $D_c = \diag{\bm{\sigma}_c}$, the scaled AL problem writes out:
\begin{equation*}
  \label{eq:auglag_subproblem_scaled}
  \min_{
    \substack{
  \bm{u}^\flat \leq \bm u \leq \bm{u}^\sharp, \\
  \bm{s}^\flat \leq \bm s \leq \bm{s}^\sharp
    }
  } \; \sigma_f \cdot f(\bm u) + \bm{y}^\top \cdot D_c \big( c(\bm u) - \bm s\big) +
  \frac{\rho}{2} \| D_c\big(c(\bm u) - \bm s \big) \|^2 \; .
\end{equation*}

\subsection{Static optimal power flow}
\label{sec:algo:static_opf}
The static OPF amounts to solving~\eqref{eq:problem_algo}
with an AL algorithm.

\subsubsection{Solving the Augmented Lagrangian's subproblems}
Starting from initial primal-dual
variables $(\bm{u}_0, \bm{s}_0; \bm{y}_0)$, the algorithm solves
at each iteration~\eqref{eq:auglag_subproblem} to a given tolerance,
for a fixed penalty $\rho$ and multiplier vector $\bm{y}$.

Here, we solve each subproblem \eqref{eq:auglag_subproblem} with an IPM.
Even if IPMs lack the inherent warmstart capability of active-set methods~\cite{nocedal_numerical_2006},
they do not require reordering of the Hessian matrix; expensive on GPUs.

We note $\bm w := (\bm u, \bm s) \in \mathbb{R}^{n_u + m}$ the primal variable,
and $\bm{z} \in \mathbb{R}^{n_u +m}$ the dual variable associated with
the bound constraints linked to $\bm w = (\bm u, \bm s)$.
The IPM solves the following unconstrained problem:
\begin{multline}
  \label{eq:ipm_subproblem}
  \min_{\bm u, \bm s} \; \psi_{\mu}(\bm w; \bm y) := L_\rho(\bm u, \bm s; \bm y)
  + B_\mu(\bm{u}, \bm{s}) \; ,
\end{multline}
with $B_\mu(\cdot)$ the barrier term
\begin{multline}
  B_\mu(\bm u, \bm s) =
  - \mu \sum_{i=1}^{n_u} \big(\log(u_i - u_i^\flat) + \log(u_i^\sharp - u_i)\big) \\
  - \mu \sum_{i=1}^{m} \big(\log(s_i - s_i^\flat) + \log(s_i^\sharp - s_i)\big)
  \; .
\end{multline}
Introducing $W = \diag{\bm w}$ and $Z = \diag{\bm z}$,
the first-order optimality
conditions of \eqref{eq:ipm_subproblem} are
\begin{equation}
  \label{eq:ipm_firstorderconditions}
  \left\{
  \begin{aligned}
    & \nabla L_\rho(\bm w; \bm y) - \bm z = 0  \\
    & WZ - \mu \bm e = 0 \; .
  \end{aligned}
  \right.
\end{equation}
The IPM algorithm solves iteratively the system of nonlinear
equations~\eqref{eq:ipm_firstorderconditions} with a Newton method.
Along the iterations, the barrier term $\mu$ is driven to 0 to recover
the first-order conditions of the original AL subproblem~\eqref{eq:auglag_subproblem}.

\subsubsection{Solving the KKT system}

By applying the Newton method to the first-order conditions~\eqref{eq:ipm_firstorderconditions}, the
descent direction $\bm d = (\bm d_w, \bm d_z)$ is computed directly
as a solution of the nonsymmetric linear system
\begin{equation}
  \label{eq:ipm_kkt}
  \begin{bmatrix}
    \nabla^2 L_\rho & -I  \\
    Z & V
  \end{bmatrix}
  \begin{bmatrix}
    \bm d_w \\ \bm d_z
  \end{bmatrix}
  = -
  \begin{bmatrix}
     \nabla L_\rho(\bm w; \bm y) - \bm z   \\
     WZ - \mu \bm e
  \end{bmatrix}
  \; .
\end{equation}
Eliminating the last block row in \eqref{eq:ipm_kkt},
the descent $\bm d_w$ is
\begin{equation}
  \label{eq:ipm_kkt_reduced}
  \big[ \nabla^2 L_\rho + \Sigma \big] \bm{d}_w
  =
  - \nabla_{\bm w} \psi_\mu(\bm w, \bm y)^\top
  \; ,
\end{equation}
with the diagonal matrix $\Sigma = W^{-1} Z$. Once $\bm{d}_w$ computed with \eqref{eq:ipm_kkt_reduced},
we recover the dual descent direction with $\bm d_z = \mu W^{-1} \bm{e} - \bm{z} - \Sigma\; \bm{d}_w$.
To solve the linear system~\eqref{eq:ipm_kkt_reduced} efficiently, we exploit the structure
of the AL's Hessian~\eqref{eq:hessian_auglag} with a Schur-complement approach.

\begin{proposition}[Schur complement]
  \label{prop:schur}
  The linear system~\eqref{eq:ipm_kkt_reduced} is equivalent to
  solving
  \begin{equation}
    \label{eq:ipm_kkt_schur}
    \left\{
    \begin{aligned}
      S_{\bm{u}\bm{u}} \bm{d}_u &=
      - \nabla_{\bm{u}} \psi_\mu^\top + \rho \nabla_{\bm{u}} c^\top \big[\Sigma_s + \rho I]^{-1} \nabla_{\bm{s}} \psi_\mu^\top
        \\
      \bm{d}_s &= \big[\Sigma_s + \rho I]^{-1}\big(- \nabla_{\bm{s}} \psi_\mu^\top + \rho (\nabla_{\bm{u}} c) \bm{d}_u \big) \; ,
    \end{aligned}
    \right.
  \end{equation}
  where $S_{\bm{u}\bm{u}}$ is the Schur-complement of the lower right block
  of the matrix $\big[\nabla^2 L_\rho + \Sigma \big]$:
  \begin{equation}
    \label{eq:schur_matrix}
    S_{\bm{u}\bm{u}} = H_{\bm{u}\bm{u}} + \Sigma_u +
    \rho (\nabla_{\bm{u}} c)^\top \big(1 - \rho \big[\Sigma_s + \rho I\big]^{-1}\big) \nabla_{\bm{u}} c
    \; .
  \end{equation}
\end{proposition}
By using Proposition~\ref{prop:schur}, we observe that the
slack descent direction $\bm{d}_s$ depends linearly on the control
descent direction $\bm{d}_u$. Instead of factorizing the regularized
Hessian matrix $\nabla^2 L_\rho + \Sigma$ (with dimension
$(n_u + m) \times (n_u +m)$), we only have to factorize the
Schur complement matrix~\eqref{eq:schur_matrix} (with
dimension $n_u \times n_u$): the complexity becomes independent
of the number of constraints in the problem.

Note that with a inertia-controlling algorithm, we can
control $\Sigma$ to ensure a positive definite matrix $\nabla^2 L_\rho + \Sigma$. It follows that $S_{\bm{u}\bm{u}}$
is positive definite~\cite[Theorem 7.7.7, p.495]{horn2012matrix}.
Thus, the Schur-complement matrix can be factorized efficiently
with a dense Cholesky factorization; readily available on GPUs.

\subsubsection{Static OPF algorithm}
The static OPF algorithm implements a typical
AL algorithm, as presented in \cite{nocedal_numerical_2006,conn2013lancelot}.
The most expensive step is the resolution
of the subproblems with IPM (all other operations are only involving vector
or scalar operations). We adapt the IPM algorithm to our AL context.
At each iteration $k$, we warmstart the IPM algorithm with the
previous primal-dual solutions $(\bm{w}_{k-1}, \bm{z}_{k-1})$.
The initial barrier is chosen according to \cite{ma2021julia}, which
decreases significantly the IPM iterations when we are close to the optimal solution.

The AL algorithm has two main bottlenecks: (i) its convergence is only linear~\cite{nocedal_numerical_2006}
(ii) at a new iterate $\bm{u}$ there is no guarantee that there exists a corresponding
state $\bm{x}$ (the Jacobian $\nabla_{\bm{x}} g$ can be singular if we leave the power flow domain).
In future work we will alleviate (i) with a refinement step~\cite{birgin2008improving}, and (ii) can be safeguarded by computing the maximum step in the line-search procedure with bifurcation analysis~\cite{dobson1993computing}.

\subsection{Real-time optimal power flow}
\label{sec:algo:rto_opf}
Like the static OPF algorithm, the real-time OPF algorithm is also derived from
the AL formulation~\eqref{eq:auglag_subproblem}.
Our real-time OPF algorithm follows the method introduced
in \cite{zavala2010real}, and does not use an active-step procedure as in \cite{tang2017real}.

In this subsection, we assume that the problem is parameterized by a time index $t \in \mathbb{N}$,
corresponding to varying load conditions $(\bm{p}_t^d, \bm{q}_t^d)$:
\begin{equation}
  \label{eq:auglag_subproblem_rto}
  \min_{
  \bm{u}, \bm{s}
  } \; f_t(\bm u) + \bm{y}^\top \big( c_t(\bm u) - \bm s\big) +
  \frac{\rho}{2} \| c_t(\bm u) - \bm s \|^2 \; .
\end{equation}
To avoid an expensive explicit solution for $\eqref{eq:auglag_subproblem_rto}$
for each time $t$, we track a suboptimal solution by reiterating the following
procedure. We initiate at time $t=0$ the real-time algorithm
with an optimal solution $\bm{w}_0^\star = (\bm{u}_0^\star, \bm{s}_0^\star)$
(for instance computed with the static OPF algorithm introduced in \refsec{sec:algo:static_opf}).
Then, for all time $t$, we update the primal-dual variables $(\bm{w}_t, \bm{y}_t)$ as follows
\begin{enumerate}
  \item For new loads $(\bm{p}_t^d, \bm{q}_t^d)$, compute
    the gradient $\bm{g}_t = \nabla_{\bm{w}} L_{\rho,t}(\bm{w}_t, \bm{y}_t)$
    and the Hessian $H_t = \nabla^2_{\bm{w}\bm{w}} L_{\rho,t}(\bm{w}_t, \bm{y}_t)$
  \item Solve the bounded quadratic problem (QP)
    and update the primal variable $\bm{w}_{t+1}$ with the solution
    \begin{equation}
      \label{eq:qp_rto}
      \begin{aligned}
        \min_{\bm{w}} \;& \bm{g}_t^\top (\bm{w} - \bm{w}_t) +
     \frac 12 (\bm{w} - \bm{w}_t)^\top H_t (\bm{w} - \bm{w}_t) \\
        \text{s.t. } \quad & \bm{w}^\flat \leq \bm{w} \leq \bm{w}^\sharp \; ,
      \end{aligned}
    \end{equation}
  \item Set $\bm{y}_{t+1} = \bm{y}_t + \rho (c_t(\bm{u}_{t+1}) - \bm{s}_{t+1})$
\end{enumerate}
Step (1) can be evaluated efficiently using the reduced space procedure
we introduced in Section~\ref{sec:projection}. The QP problem~\eqref{eq:qp_rto}
can be solved efficiently with an IPM method.

Indeed, the problem~\eqref{eq:qp_rto}
presents the same structure as the original AL subproblem~\eqref{eq:auglag_subproblem}.
That means that in practice, we can solve~\eqref{eq:qp_rto} with IPM,
using the same Schur-complement procedure introduced in Proposition~\ref{prop:schur}.
As the Hessian is constant, most of the time is spent factorizing the Schur-complement
matrix in the IPM algorithm.

\section{Numerical results}
\label{sec:results}
We now implement the reduction presented in Section~\ref{sec:projection},
and use the implementation to solve the static OPF and real-time OPF
presented in Section~\ref{sec:algorithm}.
We detail the implementation we are using in \refsec{sec:results:implementation}.
Then, the two algorithms are tested respectively in \refsec{sec:results:static_opf}
and \refsec{sec:results:rt_opf}.

\subsection{Implementation}
\label{sec:results:implementation}
The two algorithms introduced in Section~\ref{sec:algorithm} have
two main blocks: the computation of the dense reduced Hessian~\eqref{eq:hessian_auglag}, and the
factorization of the (dense) Schur-complement matrix~\eqref{eq:schur_matrix}.
Both operations are amenable to GPU accelerators, leading us to an entirely GPU accelerated implementation
in the programming language Julia publicly available \footnote{\url{https://github.com/exanauts/ExaPF-Opt}}.

\subsubsection{Implementing the callbacks on the GPU}
\label{sec:results:callbacks}
Based on the algorithm laid out in \refsec{sec:projection:discussion}
\begin{itemize}
  \item \emph{Kernels:} we use the portability layer KernelAbstractions.jl to implement the objective
    $f(\cdot)$, the power flow $g(\cdot)$, and the constraints $c(\cdot)$,
    so we can evaluate all functions using vectorization on the GPU.
  \item \emph{Automatic differentation:} we have developed a custom GPU backend
    to compute the first and second-order sensitivities of the kernels
    $f(\cdot), g(\cdot), c(\cdot)$
  \item \emph{Sparse linear system:}
    all the sparse linear systems in \S\ref{sec:projection:reduced_gradient} are solved using the LU refactorization solver {\tt cusolverRF}.
    The sparsity pattern associated with the power flow's Jacobians $\nabla_{\bm x} g$ is static, as defined
    by the network topology.
    Hence, we factorize the Jacobian during
    the presolve and transfer the factorization to the device. Every new
    Jacobian $\nabla_{\bm x} g$ is refactorized directly on the GPU, without
    any data host-device transfer. With the Jacobian refactorized,
    {\tt cusolverRF} is able to solve the linear system with multiple right-hand-sides,
    in batch. Together with our custom AD GPU backend, it enables us
    to accumulate the reduced Hessian in one shot using batched parallel Hessian-vector products~\eqref{eq:hessvecprod}.
  \item \emph{Power flow solver:} the power flow equations are solved with the Newton-Raphson
    algorithm~\eqref{eq:projection:newtonraphson}, entirely on the GPU. We stop the algorithm
    once: $\|g(\bm{x}_k, \bm{u})\|_2 < 10^{-10}$.
\end{itemize}

\subsubsection{Porting the optimization algorithms to the GPU}
As discussed in Section~\ref{sec:algorithm}, the IPM algorithm
is the core component both for the static OPF and for the real-time OPF algorithms.
Here, we use the MadNLP solver~\cite{shin2020graph}, written entirely in Julia. We modify
MadNLP to wrap the GPU callbacks. To that end, MadNLP
takes the reduced Hessian --- computed on the GPU as described in \S\ref{sec:results:callbacks}
--- assembles the Schur-complement~\eqref{eq:schur_matrix} with {\tt cuBLAS},
and then applies the Cholesky factorization of {\tt cuSOLVER}.
All that procedure happens entirely on the GPU, without any data transfer to the host.

\subsubsection{Benchmark library} We test the static and the real-time OPF algorithms
on three different instances from the PGLIB library~\cite{babaeinejadsarookolaee2019power},
depicted in Table~\ref{tab:test_instances}.
We run all experiments with a NVIDIA V100 GPU (with 32GB RAM),
using {\tt CUDA 11.3}.

\begin{table}[!ht]
  \resizebox{.5\textwidth}{!}{
    \begin{tabular}{c|cc|ccc}
      {\bf Case} & {\bf $n_v$} & {\bf $n_e$}& {\bf $n_x$} & {\bf $n_u$}& $m$ \\
      \hline
      PEGASE1354 & 1,354 & 1,991 & 2,447 & 519 & 5,337 \\
      PEGASE2869 & 2,869 & 4,582 & 5,227 & 1,019 & 12,034 \\
      PEGASE9241 & 9,241 & 16,049 & 17,036 & 2,889 & 41,340 \\
    \end{tabular}
  }
  \caption{Case instances obtained from PGLIB}
  \label{tab:test_instances}
\end{table}

\subsection{Static optimal power flow}
\label{sec:results:static_opf}
We test the static OPF algorithm presented in \S\ref{sec:algo:static_opf}
on the three instances in Table~\ref{tab:results:static_opf}. In the stopping criterion,
we set $\eta_{primal} = 10^{-5}$ and $\eta_{dual} = 10^{-4}$.
The results are presented in Table~\ref{tab:results:static_opf}
(the reference objective $f^\star$ and computation time $t_{ref}$ are computed by solving the original OPF problem~\eqref{eq:acopf}
with Ipopt+MA27 and PowerModels.jl~\cite{coffrin2018powermodels}). We observe that we are able to recover Ipopt's solution, with good accuracy.
The algorithm takes many iterations to converge, leading to total running time
being an order of magnitude greater than Ipopt (on {\tt case9241pegase}, Ipopt converges in only 60s).
However, our algorithm remains tractable, and is a net improvement comparing to
previous implementation of the reduced space methods~\cite{kardos2020reduced}.

We detail in Figure~\ref{fig:convergence_auglag} the convergence of the algorithm
on {\tt case1354pegase}. The algorithm converges in 369 iterations. The IPM algorithm
is restarted 18 times in the AL routine, to update
the penalty $\rho$ and the multiplier $\bm{y}$ (thus explaining the peaks in the evolution of the dual feasibility).
As we increase the penalty $\rho$, the primal infeasibility decreases linearly.
The second plot depicts the evolution of the maximum relative violation for
each operational constraints (in p.u). We observe that in practice
the evolution of the infeasibility depends on the
scaling used for the different constraints (\S\ref{sec:algo:scaling}): the greater
the scaling $\sigma_{c,i}$ is, the faster the infeasibility of the constraint $i$
will be driven to 0.

\begin{table}[!ht]
  \resizebox{.5\textwidth}{!}{
    \begin{tabular}{c|ccc|ccc}
      {\bf Case} &  $f$ & $(f - f^\star)/f^\star$ & $t_{ref}$ (s) &Time AL (s) & \#it & Time/it (s) \\
      \hline
      PEGASE1354 & $7.4069287 \times 10^{4}$ & $9.2 \times 10^{-7}$ & 2.5 & 48. & 369 & 0.13 \\
      PEGASE2869 &  $1.3399920 \times 10^{5}$ & $6.7 \times 10^{-7}$ & 6.3 & 134. & 427 & 0.31 \\
      PEGASE9241 & $3.1591213 \times 10^{5}$ & $3.4 \times 10^{-6}$ & 60. &1495. & 672 & 2.22
    \end{tabular}
  }
  \caption{Performance of the static OPF algorithm.}
  \label{tab:results:static_opf}
\end{table}


\begin{figure}[!ht]
  \centering
  \includegraphics[width=0.5\textwidth]{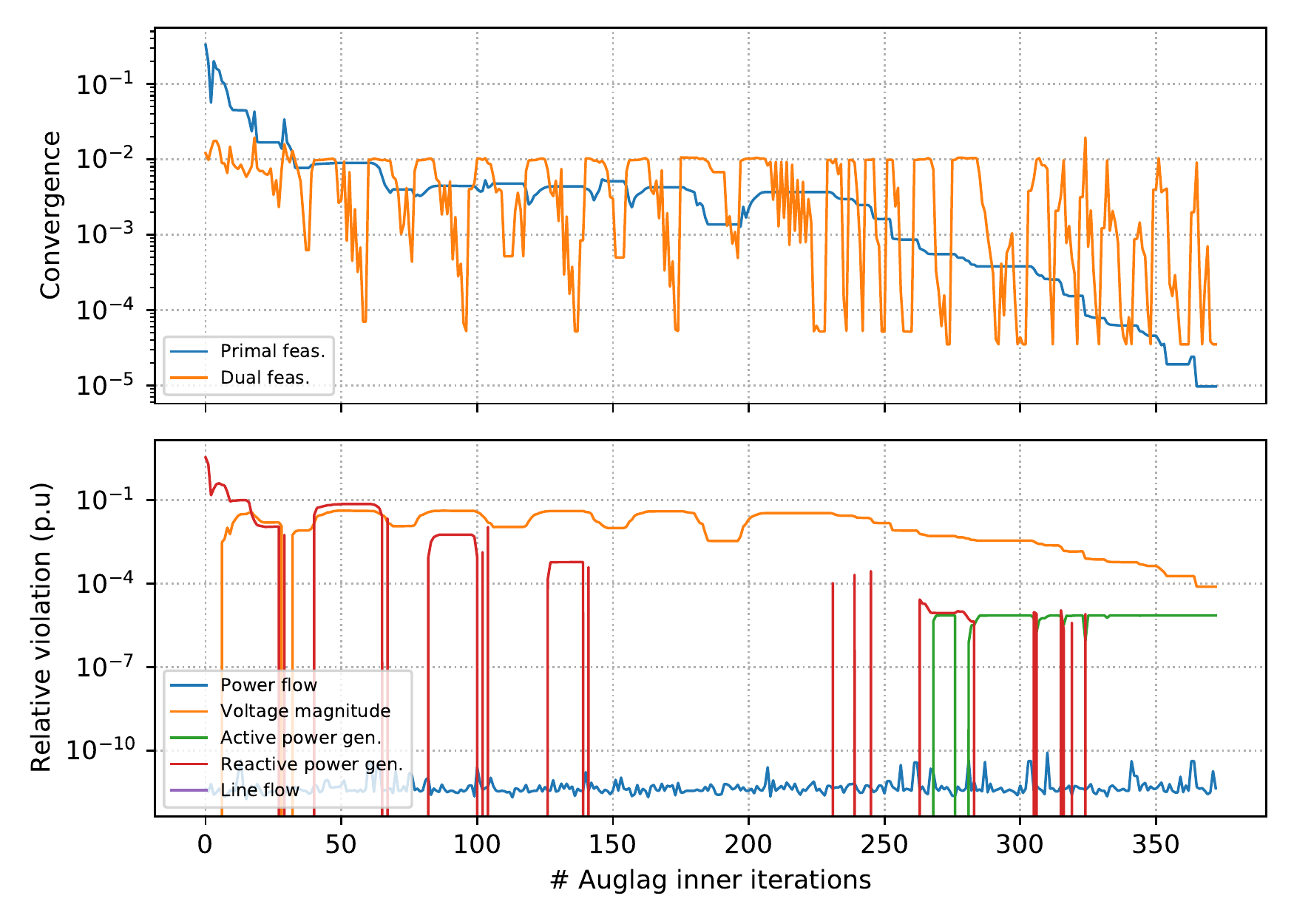}
  \caption{Solving the static OPF for {\tt case1354pegase}: the first plot displays the evolution
  of the primal and the dual infeasibility, and the second plot displays the
  evolution of the relative infeasibility for the different kind of operational constraints.
  }
  \label{fig:convergence_auglag}
\end{figure}

\subsection{Real-time optimal power flow}
\label{sec:results:rt_opf}
We test the real-time OPF algorithm presented in \S\ref{sec:algo:rto_opf}.
We suppose that the loads $(\bm{p}_{t}^g, \bm{q}_g^t)$
are varying along the time, and are updated at every minute. To model the evolution of the loads, we use the time-series
provided by~\cite{kim2020real}: as illustrated in Figure~\ref{fig:demand}, all loads drop suddenly by 20\% at time $t=2$, making it difficult to track the optimal solution. For all time $t$, the real-time algorithm should
update the tracking control $\bm{w}$ in a time-span $\Delta t$, with $\Delta t \ll 1\text{mn}$.
\begin{figure}[!ht]
  \centering
  \includegraphics[width=0.5\textwidth]{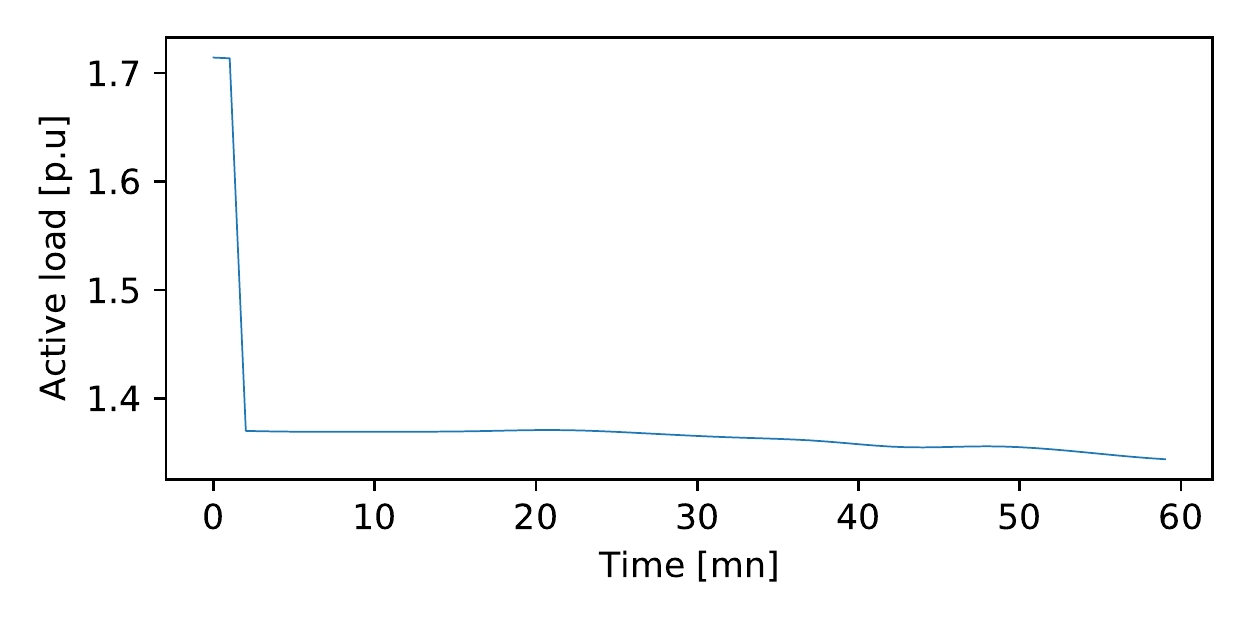}
  \caption{
    Evolution of the active load at bus $1$ for case1354pegase.
  }
  \label{fig:demand}
\end{figure}

At time $t=0$, the initial primal-dual solution $(\bm{w}_0, \bm{y}_0)$ is computed
using the static OPF algorithm.
For all time $t$, we use our GPU implementation to update the tracking point.
This operation involves two expensive operations: (i) computing the reduced
Hessian for the new load conditions $(\bm{p}_t^d, \bm{q}_t^d)$ (ii) solving the bounded QP problem~\eqref{eq:qp_rto}.
We detail in Table~\ref{tab:results:time_rto_opf} the time spent inside these two operations:
{\tt LinAlg} is the time spent in the Cholesky solver (factorization and triangular solve).
{\tt Eval } is the time spent in the QP callbacks (involving only BLAS operations).
We observe that it takes less than one second to update the tracking point
for {\tt case1354pegase} and {\tt case9241pegase}, slightly longer for {\tt case9241pegase}.

\begin{table}[!ht]
  \resizebox{.5\textwidth}{!}{
    \begin{tabular}{c|c|cccc|c}
      & & \multicolumn{4}{c}{QP resolution} & \\
      \hline
      {\bf Case} &  Hess (s) & \#it & LinAlg. (s) & Eval. (s) & Tot (s) & $\Delta t$ (s)\\
      \hline
      PEGASE1354 & 0.06 & 12 & 0.005 & 0.014 & 0.08 & 0.14 \\
      PEGASE2869 & 0.16 & 12 & 0.096 & 0.029 & 0.16 & 0.32 \\
      PEGASE9241 & 1.6 & 14 & 0.187 & 0.200 & 3.0 & 3.6
    \end{tabular}
  }
  \caption{
    Time to update the tracking point (in seconds).
  }
  \label{tab:results:time_rto_opf}
\end{table}

We now assess the effectiveness of the tracking algorithm. In Figure~\ref{fig:rto_algo}, we control {\tt case1354pegase}
every minute, during one hour. At time $t=2$, the perturbation happens and the loads drop by 20\%. The real-time algorithm is able
to recover a suboptimal solution in 15 minutes (at $t>17$, the relative gap with Ipopt's objective is less than $2\times10^{-4}$, with a primal infeasibility for the operational constraints closes to $1\times10^{-2}$). The third plot, displaying the absolute difference between $\bm{p}_t^g$ and the optimal point
$\bm{p}_t^{g,\star}$, shows that the median deviation in active power generations is less than $1 \times 10^{2}$.

\begin{figure}[!ht]
  \centering
  \includegraphics[width=0.5\textwidth]{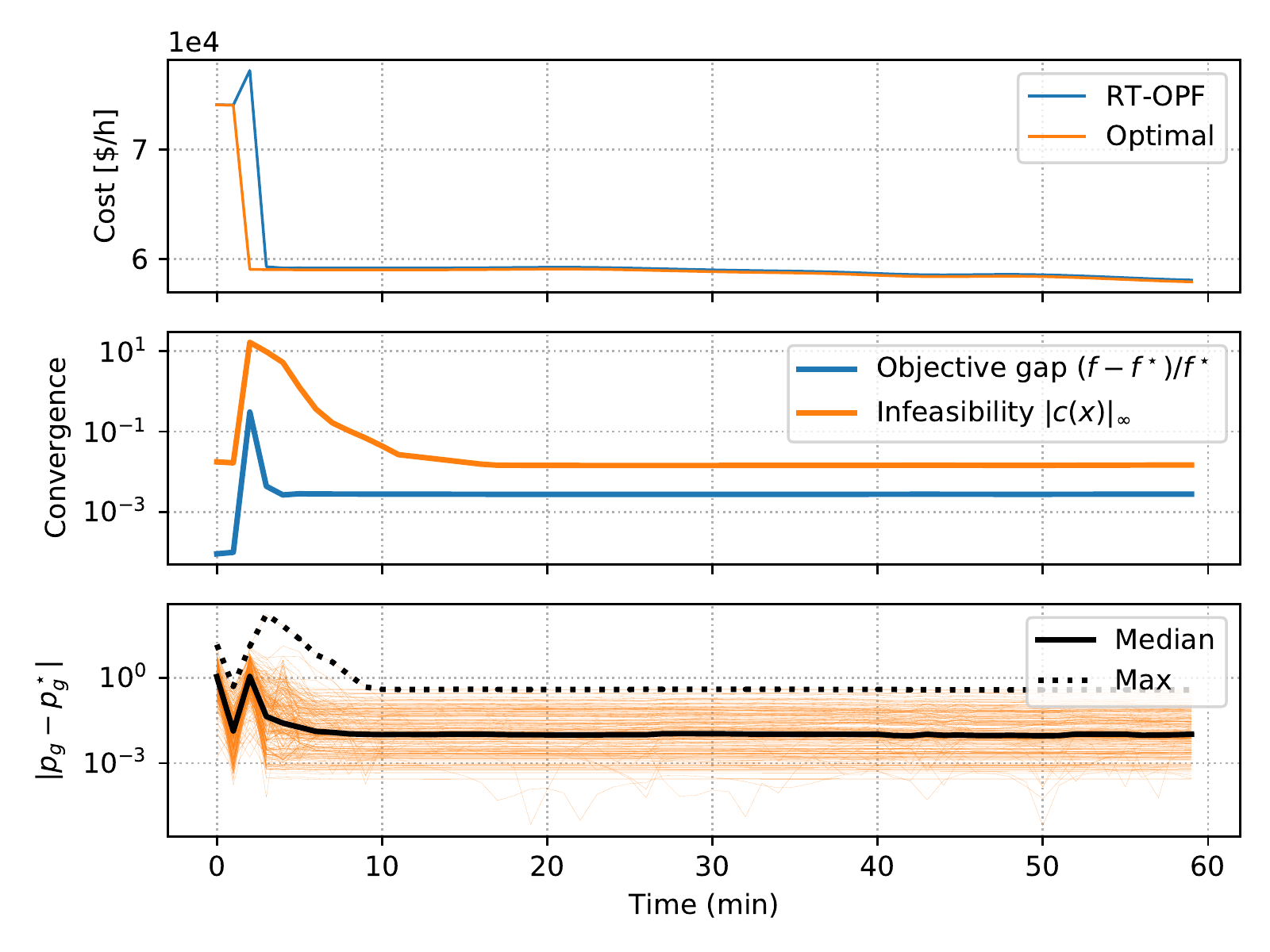}
  \caption{
    Tracking the AC-OPF solution of {\tt case1354pegase} with the RT-OPF
    algorithm. The first and the second plots display respectively
    (i) the operating costs computed respectively with the real-time OPF and with Ipopt.
    (ii) the convergence of the real-time OPF.
    The third plot displays for each generator the absolute difference between Ipopt solution and real-time OPF's setpoint.
  }
  \label{fig:rto_algo}
\end{figure}

\section{Conclusion}
We have devised a feasible Augmented Lagrangian algorithm, whose
iterates satisfy by design the physical equations of the power network.
A GPU implementation of the reduced space algorithm has proven to be critical
for its tractability. On the one hand, we have shown that
the algorithm is effective at solving large-scale static OPF problem and, to the best of
our knowledge, is the first reduced space algorithm able to solve {\tt case9241pegase}
with second-order information. On the other hand, the reduced space algorithm is able to track a suboptimal solution in a real-time OPF setting, and can adapt quickly to large variations in the loads.

Future work will optimize the algorithm by intertwining more
closely the IPM algorithm with the Augmented Lagrangian algorithm with the expectation
of more accurate solutions, in less iterations.


\bibliographystyle{IEEEtran}
\bibliography{biblio.bib}

\begin{thebibliography}{10}
\providecommand{\url}[1]{#1}
\csname url@samestyle\endcsname
\providecommand{\newblock}{\relax}
\providecommand{\bibinfo}[2]{#2}
\providecommand{\BIBentrySTDinterwordspacing}{\spaceskip=0pt\relax}
\providecommand{\BIBentryALTinterwordstretchfactor}{4}
\providecommand{\BIBentryALTinterwordspacing}{\spaceskip=\fontdimen2\font plus
\BIBentryALTinterwordstretchfactor\fontdimen3\font minus
  \fontdimen4\font\relax}
\providecommand{\BIBforeignlanguage}[2]{{%
\expandafter\ifx\csname l@#1\endcsname\relax
\typeout{** WARNING: IEEEtran.bst: No hyphenation pattern has been}%
\typeout{** loaded for the language `#1'. Using the pattern for}%
\typeout{** the default language instead.}%
\else
\language=\csname l@#1\endcsname
\fi
#2}}
\providecommand{\BIBdecl}{\relax}
\BIBdecl

\bibitem{Capitanescu2016}
\BIBentryALTinterwordspacing
F.~Capitanescu, ``Critical review of recent advances and further developments
  needed in {AC} optimal power flow,'' \emph{Electric Power Systems Research},
  vol. 136, pp. 57--68, Jul. 2016. [Online]. Available:
  \url{https://doi.org/10.1016/j.epsr.2016.02.008}
\BIBentrySTDinterwordspacing

\bibitem{tang2017real}
Y.~Tang, K.~Dvijotham, and S.~Low, ``Real-time optimal power flow,'' \emph{IEEE
  Transactions on Smart Grid}, vol.~8, no.~6, pp. 2963--2973, 2017.

\bibitem{gan2016online}
L.~Gan and S.~H. Low, ``An online gradient algorithm for optimal power flow on
  radial networks,'' \emph{IEEE Journal on Selected Areas in Communications},
  vol.~34, no.~3, pp. 625--638, 2016.

\bibitem{hauswirth2017online}
A.~Hauswirth, A.~Zanardi, S.~Bolognani, F.~D{\"o}rfler, and G.~Hug, ``Online
  optimization in closed loop on the power flow manifold,'' in \emph{2017 IEEE
  Manchester PowerTech}.\hskip 1em plus 0.5em minus 0.4em\relax IEEE, 2017, pp.
  1--6.

\bibitem{dall2016optimal}
E.~Dall’Anese and A.~Simonetto, ``Optimal power flow pursuit,'' \emph{IEEE
  Transactions on Smart Grid}, vol.~9, no.~2, pp. 942--952, 2016.

\bibitem{dommel_optimal_1968}
H.~Dommel and W.~Tinney, ``\BIBforeignlanguage{en}{Optimal {Power} {Flow}
  {Solutions}},'' \emph{\BIBforeignlanguage{en}{IEEE Transactions on Power
  Apparatus and Systems}}, vol. PAS-87, no.~10, pp. 1866--1876, Oct. 1968.

\bibitem{zavala2010real}
V.~M. Zavala and M.~Anitescu, ``Real-time nonlinear optimization as a
  generalized equation,'' \emph{SIAM Journal on Control and Optimization},
  vol.~48, no.~8, pp. 5444--5467, 2010.

\bibitem{chiang2017augmented}
N.-Y. Chiang, R.~Huang, and V.~M. Zavala, ``An augmented lagrangian filter
  method for real-time embedded optimization,'' \emph{IEEE Transactions on
  Automatic Control}, vol.~62, no.~12, pp. 6110--6121, 2017.

\bibitem{frank_optimal_2012}
\BIBentryALTinterwordspacing
S.~Frank, I.~Steponavice, and S.~Rebennack, ``\BIBforeignlanguage{en}{Optimal
  power flow: a bibliographic survey {I}: {Formulations} and deterministic
  methods},'' \emph{\BIBforeignlanguage{en}{Energy Systems}}, vol.~3, no.~3,
  pp. 221--258, Sep. 2012. [Online]. Available:
  \url{http://link.springer.com/10.1007/s12667-012-0056-y}
\BIBentrySTDinterwordspacing

\bibitem{kardos2020reduced}
J.~Kardos, D.~Kourounis, and O.~Schenk, ``Reduced-space interior point methods
  in power grid problems,'' \emph{arXiv preprint arXiv:2001.10815}, 2020.

\bibitem{cain2012history}
M.~B. Cain, R.~P. O’neill, A.~Castillo \emph{et~al.}, ``History of optimal
  power flow and formulations,'' \emph{Federal Energy Regulatory Commission},
  vol.~1, pp. 1--36, 2012.

\bibitem{tinney1967power}
W.~F. Tinney and C.~E. Hart, ``Power flow solution by newton's method,''
  \emph{IEEE Transactions on Power Apparatus and systems}, no.~11, pp.
  1449--1460, 1967.

\bibitem{wang1992second}
Z.~Wang, I.~M. Navon, F.-X. Le~Dimet, and X.~Zou, ``The second order adjoint
  analysis: theory and applications,'' \emph{Meteorology and atmospheric
  physics}, vol.~50, no.~1, pp. 3--20, 1992.

\bibitem{rockafellar1976augmented}
R.~T. Rockafellar, ``Augmented lagrangians and applications of the proximal
  point algorithm in convex programming,'' \emph{Mathematics of operations
  research}, vol.~1, no.~2, pp. 97--116, 1976.

\bibitem{bertsekas1976multiplier}
D.~P. Bertsekas, ``Multiplier methods: a survey,'' \emph{Automatica}, vol.~12,
  no.~2, pp. 133--145, 1976.

\bibitem{wachter2006implementation}
A.~W{\"a}chter and L.~T. Biegler, ``On the implementation of an interior-point
  filter line-search algorithm for large-scale nonlinear programming,''
  \emph{Mathematical programming}, vol. 106, no.~1, pp. 25--57, 2006.

\bibitem{birgin2014practical}
E.~G. Birgin and J.~M. Mart{\'\i}nez, \emph{Practical augmented Lagrangian
  methods for constrained optimization}.\hskip 1em plus 0.5em minus 0.4em\relax
  SIAM, 2014.

\bibitem{nocedal_numerical_2006}
J.~Nocedal and S.~J. Wright, \emph{\BIBforeignlanguage{en}{Numerical
  optimization}}, 2nd~ed., ser. Springer series in operations research.\hskip
  1em plus 0.5em minus 0.4em\relax New York: Springer, 2006, oCLC: ocm68629100.

\bibitem{horn2012matrix}
R.~A. Horn and C.~R. Johnson, \emph{Matrix analysis}.\hskip 1em plus 0.5em
  minus 0.4em\relax Cambridge university press, 2012.

\bibitem{conn2013lancelot}
A.~R. Conn, G.~Gould, and P.~L. Toint, \emph{LANCELOT: a Fortran package for
  large-scale nonlinear optimization (Release A)}.\hskip 1em plus 0.5em minus
  0.4em\relax Springer Science \& Business Media, 2013, vol.~17.

\bibitem{ma2021julia}
D.~Ma, D.~Orban, and M.~A. Saunders, ``{A Julia implementation of Algorithm NCL
  for constrained optimization},'' \emph{arXiv preprint arXiv:2101.02164},
  2021.

\bibitem{birgin2008improving}
E.~G. Birgin and J.~M. Mart{\'\i}nez, ``Improving ultimate convergence of an
  augmented lagrangian method,'' \emph{Optimization Methods and Software},
  vol.~23, no.~2, pp. 177--195, 2008.

\bibitem{dobson1993computing}
I.~Dobson, ``Computing a closest bifurcation instability in multidimensional
  parameter space,'' \emph{Journal of nonlinear science}, vol.~3, no.~1, pp.
  307--327, 1993.

\bibitem{shin2020graph}
S.~Shin, C.~Coffrin, K.~Sundar, and V.~M. Zavala, ``Graph-based modeling and
  decomposition of energy infrastructures,'' \emph{arXiv preprint
  arXiv:2010.02404}, 2020.

\bibitem{babaeinejadsarookolaee2019power}
S.~Babaeinejadsarookolaee, A.~Birchfield, R.~D. Christie, C.~Coffrin,
  C.~DeMarco, R.~Diao, M.~Ferris, S.~Fliscounakis, S.~Greene, R.~Huang
  \emph{et~al.}, ``The power grid library for benchmarking ac optimal power
  flow algorithms,'' \emph{arXiv preprint arXiv:1908.02788}, 2019.

\bibitem{coffrin2018powermodels}
C.~Coffrin, R.~Bent, K.~Sundar, Y.~Ng, and M.~Lubin, ``Powermodels. jl: An
  open-source framework for exploring power flow formulations,'' in \emph{2018
  Power Systems Computation Conference (PSCC)}.\hskip 1em plus 0.5em minus
  0.4em\relax IEEE, 2018, pp. 1--8.

\bibitem{kim2020real}
Y.~Kim and M.~Anitescu, ``A real-time optimization with warm-start of
  multiperiod ac optimal power flows,'' \emph{Electric Power Systems Research},
  vol. 189, p. 106721, 2020.

\end{thebibliography}
%

\vfill
\begin{flushright}
{\footnotesize
  \framebox{\parbox{0.5\textwidth}{
The submitted manuscript has been created by UChicago Argonne, LLC,
Operator of Argonne National Laboratory (``Argonne"). Argonne, a
U.S. Department of Energy Office of Science laboratory, is operated
under Contract No. DE-AC02-06CH11357. The U.S. Government retains for
itself, and others acting on its behalf, a paid-up nonexclusive,
irrevocable worldwide license in said article to reproduce, prepare
derivative works, distribute copies to the public, and perform
publicly and display publicly, by or on behalf of the Government.
The Department of
Energy will provide public access to these results of federally sponsored research in accordance
with the DOE Public Access Plan. http://energy.gov/downloads/doe-public-access-plan. }}
\normalsize
}
\end{flushright}

\end{document}